\documentclass[11pt,reqno]{amsart}
\usepackage{amssymb,latexsym}
\usepackage{graphicx}

\theoremstyle{plain} \numberwithin{equation}{section}
\newtheorem{thm}{Theorem}[section]
\newtheorem{theorem}[thm]{Theorem}

\newtheorem{lemma}[thm]{Lemma}

\newtheorem{proposition}[thm]{Proposition}

\begin{document}

\setcounter{page}{1}

\title[Certain Binomial Sums with recursive coefficients]{Certain Binomial Sums with recursive coefficients}
\author{Emrah Kilic}

\address{TOBB\ University of Economics and Technology Mathematics\\
               Department 06560 S\"{o}g\"{u}t\"{o}z\"{u}\\
               Ankara, Turkey} \email{ekilic@etu.edu.tr}

%\thanks{Research supported in part by the Natural Sciences and Engineering Research Council of Canada and by Emperor Frederick II of Sicily.}
\author{Eugen J. Ionascu}

\address{Department of Mathematics\\
               Columbus State University\\
               4225 University Avenue\\
               Columbus, GA
31907, United States} \email{ionascu$\_$eugen@colstate.edu}

\begin{abstract}
In this short note, we establish some identities containing sums
of binomials with coefficients satisfying third order linear
recursive relations. As a result and in particular, we obtain
general forms of earlier identities involving binomial
coefficients and Fibonacci type sequences.
\end{abstract}

\maketitle

\section{Introduction}

There are many types of identities containing sums of certain
functions of binomial coefficients and Fibonacci, Lucas or Pell
numbers. Let us give a
few examples of such identities (see \cite{D,V}):%
\begin{eqnarray}
\sum_{k=0}^{n}\tbinom{n}{k}F_{k}=F_{2n},\ \ \sum_{k=0}^{n}\tbinom{n}{k}%
2^{k}F_{k}=F_{3n},\sum_{k=0}^{2n}\tbinom{2n}{k}F_{2k}
&=&5^{n}F_{2n}
\label{1} \\
n\geq m,\ \sum_{k=0}^{n}\left( -1\right) ^{k}\tbinom{n}{k}%
F_{n+k-m}=F_{n-m},\ \ \ \sum_{k=0}^{2n}\tbinom{2n}{k}F_{2k}^{2}
&=&5^{n-1}L_{2n},\   \label{2} \\
\sum_{k=0}^{2n}\tbinom{2n}{k}L_{2k}=5^{n}L_{2n},\ \ \ \
\sum_{k=0}^{2n}\left( -1\right) ^{k}\tbinom{2n}{k}2^{k-1}L_{k}
&=&5^{n}, \label{3}
\end{eqnarray}%
\noindent where $F_{n}$ and $L_{n}$ stand as usual for the
$n^{th}$ Fibonacci and respectively the $n^{th}$ Lucas number. We
remind the reader that $F_{0}=0$, $F_{1}=1$, $L_{0}=2$, $L_{1}=1$
and $F_{n+1}=F_{n}+F_{n-1}$, $L_{n+1}=L_{n}+L_{n-1}$ for $n\in
\mathbb{N}$. As a more sophisticated example the following
identity was asked by Hoggatt as an advanced problem
in \cite{H}:%
\begin{equation}
\sum_{k=0}^{n}\tbinom{n}{k}F_{4mk}=L_{2m}^{n}F_{2mn}.  \label{4.1}
\end{equation}

\noindent Generalizations of the above identities appeared in
\cite{C,P}. For instance, if $u$, $v$ and $r$ are integers,
$uv(u-v)\not=0$, then
\begin{equation*}
F_{v}^{n}F_{un+r}=\sum_{k=0}^{n}\left( -1\right) ^{\left( n-k\right)u }%
\tbinom{n}{k}F_{v-u}^{n-k}F_{u}^{k}F_{vk+r},
\end{equation*}%
which generalizes the identities (\ref{1}-\ref{2}). Similar identities to (%
\ref{1}-\ref{3}) for the Pell numbers can be derived. Our interest
here is for identities in which only half of the binomial
coefficients are used. Three of such identities are

\begin{eqnarray}
\sum_{k=0}^n\tbinom{2n}{n+k}F_{k}^2&=&5^{n-1},\ \ n\in \mathbb{N},
\label{new1} \\
\sum_{k=0}^n\tbinom{2n}{n+k}L_{k}^2&=&5^{n}+2\tbinom{2n}{n} ,\ \
n\in
\mathbb{N},  \label{new2} \\
\sum_{k=0}^n\tbinom{2n}{n+k}P_{k}^2&=&8^{n-1}, \ \ n\in
\mathbb{N}, \label{new3}
\end{eqnarray}

\noindent where $P_n$ is the the $n^{th}$ Pell number ($P_0=0$,
$P_1=1$, and $P_{n+1}=2P_n+P_{n-1}$, $n\in \mathbb{N}$). We are
also going to work with generalized Fibonacci ($\{u_n\}$) and
Lucas ($\{v_n\}$) sequences defined by
\begin{equation}  \label{defofukandvk}
u_{n+1}=pu_{n}+u_{n-1},\ v_{n+1}=pv_{n}+v_{n-1} \ \ n\in
\mathbb{N},
\end{equation}

\noindent where $u_{0}=0,~u_{1}=1$ or $v_{0}=2,~v_{1}=p$ for any
complex number $p.$

One can derive easily the Binet formulae for $\left\{ u_{n}\right\} $ and $%
\left\{ v_{n}\right\} $:
\begin{equation*}
u_{n}=\frac{\alpha ^{n}-\beta ^{n}}{\alpha -\beta }\text{ and
}v_{n}=\alpha ^{n}+\beta ^{n}, \ n\in \mathbb{N}\cup \{0\},
\end{equation*}
\noindent where $\alpha =\left( p+\sqrt{p^{2}+4}\right) /2$ and $%
\beta=\left( p-\sqrt{p^{2}+4}\right) /2 $, with the principal
branch of the square root
$\sqrt{re^{i\theta}}=\sqrt{r}e^{\frac{\theta }{2}}$, $\theta \in
(-\pi,\pi]$, $r\ge 0$. From these formulae, we can easily get the
following identity:
\begin{equation}
v_{n+m}-\left( -1\right) ^{m}v_{n-m}=\left( p^{2}+4\right)
u_{n}u_{m}, \ \ \ \ m, n\in \mathbb{N}\cup \{0\}.  \label{4}
\end{equation}

As aims of this note we will establish generalizations of the formulae (\ref%
{new1}), (\ref{new2}), (\ref{new3}) and a version of (\ref{4})
with combinatorial coefficients involved. Our techniques are
definitely pure
computational and very much in line with the standard ones used to show (\ref%
{1})-(\ref{3}). One way of generalizing the identities in (\ref{new1})-(\ref%
{new3}) is to use different powers for the recursive function
term:

\begin{eqnarray}
\sum_{k=0}^{n}\tbinom{2n}{n+k}F_{k}^{4} &=&\frac{1}{25}\left(
3^{2n}-4\left(
-1\right) ^{n}+3\times 2^{2n}\right) ,\ \ n\in \mathbb{N},  \label{new12} \\
\sum_{k=0}^{n}\tbinom{2n}{n+k}L_{k}^{4} &=&3^{2n}+3\times 2^{2n}+8\binom{2n}{%
n}+4\left( -1\right) ^{n},\ \ n\in \mathbb{N},  \label{new22} \\
\sum_{k=0}^{n}\tbinom{2n}{n+k}P_{k}^{4} &=&\frac{1}{64}\left(
6^{2n}-2^{2n+2}-4\left( -1\right) ^{n}+3\times 2^{2n}\right) ,\ \
n\in \mathbb{N},  \label{new32}
\end{eqnarray}

An interesting question at this point is whether or not one can
arrange so that for some Fibonacci type recurrent sequence, powers
of any positive integer could be represented as in (\ref{new1})
and (\ref{new3}). We will address this question in the last
section of this note.

\section{Half of the binomial formula}
In this section we build up the main ingredients for our
calculations. Let us define the function $f$ of $a\in
\mathbb{C}\setminus \{0\}$ and $n\in \mathbb{N}$:
\begin{equation*}
f\left( n,a\right) =\sum\limits_{k=0}^{n}\binom{2n}{n+k}\left(
a^{k}+a^{-k}\right) .
\end{equation*}

\noindent We have the following lemma.

\begin{lemma}
\label{lem1} For every non-negative integer $n$ and $a\in \mathbb{C}%
\setminus \{0\}$ the function $f$ satisfies
\begin{equation*}
f\left( n,a\right) =\frac{1}{a^{n}}\left( a+1\right)
^{2n}+\tbinom{2n}{n}.
\end{equation*}
\end{lemma}

\begin{proof}
We can write
\begin{eqnarray*}
a^{n}f\left( n,a\right) &=&\tbinom{2n}{n}a^{n}+\tbinom{2n}{n+1}%
a^{n+1}+\ldots +\tbinom{2n}{2n}a^{2n} \\
&&+\tbinom{2n}{n}a^{n}+\tbinom{2n}{n+1}a^{n-1}+\cdots
+\tbinom{2n}{2n}a^{0}
\\
&=&\left( \tbinom{2n}{n}a^{n}+\tbinom{2n}{n+1}a^{n+1}+\ldots +\tbinom{2n}{2n}%
a^{2n}\right) \\
&&+\left( \tbinom{2n}{n-1}a^{n-1}+\cdots
+\tbinom{2n}{0}a^{0}\right) +
\tbinom{2n}{n}a^{n} \\
&=&\left( a+1\right) ^{2n}+\tbinom{2n}{n}a^{n}.
\end{eqnarray*}%
Hence the identity claimed follows by dividing by $a^n$.
\end{proof}

\noindent An observation here is necessary. We formulate this as a
proposition.

\begin{proposition}
\label{cor1} For every non-negative integer $n$ and $a\in \mathbb{C}%
\setminus \{0\}$ the following formulae yield true:
\begin{equation*}
\sum\limits_{k=0}^{n}\binom{2n}{n+k}=2^{2n-1}+\frac{1}{2}\tbinom{2n}{n},
\ \
\sum\limits_{k=0}^{n}(-1)^k\binom{2n}{n+k}=\frac{1}{2}\tbinom{2n}{n},
\end{equation*}
\begin{equation*}
f(n,(-1)^r)=(1+(-1)^r)2^{2n-1}+\tbinom{2n}{n},\ f(n,-a^2)=(-1)^n\left(a-%
\frac{1}{a}\right)^{2n}+\tbinom{2n}{n}.
\end{equation*}
\end{proposition}

\begin{proof}
To obtain the first two identities we set $a=1$ and then $a=-1$.
The last claim is obtained by substituting $a$ with $-a$ in
Lemma~\ref{lem1}.
\end{proof}

\section{Proof of the claimed identities}

We will work with the generalized Fibonacci type sequences $\{u_k\}$ and $%
\{v_k\}$ defined in the introduction. The formulae (\ref{new1}), (\ref{new12}%
), (\ref{new3}), and (\ref{new32}) are contained in the next
theorem.

\begin{theorem}
\label{thm-p}For $n\in \mathbb{N}\cup \{0\}$ and $r\in \mathbb{N}$, and $%
\{u_k\}$ and $\{v_k\}$ defined as before by (\ref{defofukandvk}),
we have
\begin{equation*}
\sum\limits_{k=0}^{n}\binom{2n}{n+k}u_{k}^{2r}=\displaystyle
\begin{cases}
\displaystyle \frac{1}{\left(p^{2}+4\right)^r} \left(\binom{2r}{r}%
2^{2n-2}+\sum_{i=0}^{r-1}
(-1)^{i(n+1)}\binom{2r}{i}v_{r-i}^{2n}\right)\ if
\ r\ is \ even. \\
\\
\displaystyle \left(p^{2}+4\right)^{n-r} \left(\sum_{i=0}^{r-1} (-1)^{i(n+1)}%
\binom{2r}{i}u_{r-i}^{2n}\right)\ if \ r\ is \ odd.%
\end{cases}%
\end{equation*}
\end{theorem}

\begin{proof}
We expand first $u_k$ to the power $2r$ using the binomial
formula:

\begin{equation*}
\sum\limits_{k=0}^{n}\binom{2n}{n+k}u_{k}^{2r}=\frac{1}{\left(
\alpha-\beta\right) ^{2r}}\sum\limits_{k=0}^{n}\binom{2n}{n+k}\left[%
\sum_{i=0}^{2r}\left( -1\right) ^{i}\binom{2r}{i}\alpha^{(2r-i)k}\beta^{ik}%
\right].
\end{equation*}

Taking into account that $\alpha ^{k}\beta ^{k}=(-1)^{k}$ and the
facts that
$\binom{2r}{i}=\binom{2r}{2r-i}$, $(-1)^{i}=(-1)^{2r-i}$, for $%
i=0,1,2,\ldots ,2r$, we can turn the above into%
\begin{eqnarray*}
\sum\limits_{k=0}^{n}\binom{2n}{n+k}u_{k}^{2r} &=&\frac{1}{\left(
p^{2}+4\right) ^{r}}\sum\limits_{k=0}^{n}\binom{2n}{n+k}\left[ (-1)^{r(1+k)}%
\binom{2r}{r}\right. \\
&&~\ \ \ \ \ \ \ \ \ \ \ \ \ \ \ \ \ \ \ \ +\sum_{i=0}^{r-1}\left(
-1\right) ^{i+ik}\left. \binom{2r}{i}(\alpha ^{2(r-i)k}+\alpha
^{-2(r-i)k})\right] .
\end{eqnarray*}

Commuting the two summations and using Lemma~\ref{lem1} (Proposition~\ref{cor1}%
) the calculation can be continued to

\begin{equation*}
\begin{array}{l}
\displaystyle\sum\limits_{k=0}^{n}\binom{2n}{n+k}u_{k}^{2r}= \\
\dfrac{1}{\left( p^{2}+4\right) ^{r}}\left( \frac{(-1)^{r}}{2}\right. \binom{%
2r}{r}f(n,(-1)^{r})+\sum\limits_{i=0}^{r-1}(-1)^{i}\left. \binom{2r}{i}%
f(n,(-1)^{i}\alpha ^{2(r-i)})\right) = \\
\\
\displaystyle\frac{1}{\left( p^{2}+4\right) ^{r}}\left( \frac{(-1)^{r}}{2}%
\right. \binom{2r}{r}f(n,(-1)^{r})+\sum\limits_{i=0}^{r-1}\left(
-1\right)
^{i}\binom{2r}{i}[\tbinom{2n}{n} \\
~\ \ \ \ \ \ \ \ \ \ \ \ \ \ \ \ \ \ \ \ \ \ \ \ \ \ \ \ \ \ \ \ \
+(-1)^{in}\left. \left( (-1)^{r-i}\beta ^{r-i}+(-1)^{i}\alpha
^{r-i}\right)
^{2n}]\right) = \\
\\
\displaystyle%
\begin{cases}
\displaystyle\frac{1}{\left( p^{2}+4\right) ^{r}}\left( \binom{2r}{r}%
2^{2n-1}+\sum_{i=0}^{r-1}(-1)^{i(n+1)}\binom{2r}{i}v_{r-i}^{2n}\right)
\ if\
r\ is\ even. \\
\\
\displaystyle\left( p^{2}+4\right) ^{n-r}\left( \sum_{i=0}^{r-1}(-1)^{i(n+1)}%
\binom{2r}{i}u_{r-i}^{2n}\right) \ if\ r\ is\ odd.%
\end{cases}%
\end{array}%
\end{equation*}
\end{proof}

The similar result to Theorem~\ref{thm-p} but for $\{v_n\}$ is
stated next.

\begin{theorem}
\label{thm-v}For $n\in \mathbb{N}\cup \{0\}$ and $r\in \mathbb{N}$, and $%
\{u_{k}\}$ and $\{v_{k}\}$ defined as before by
(\ref{defofukandvk}), we have
\begin{equation*}
\sum\limits_{k=0}^{n}\binom{2n}{n+k}v_{k}^{2r}=\displaystyle%
\begin{cases}
\displaystyle\binom{2r}{r}2^{2n-1}+2^{2r-1}\tbinom{2n}{n}%
+\sum_{i=0}^{r-1}(-1)^{in}\binom{2r}{i}v_{r-i}^{2n}\ if\ r\ is\ even. \\
\\
\displaystyle2^{2r-1}\tbinom{2n}{n}+\left( p^{2}+4\right)
^{n}\left(
\sum_{i=0}^{r-1}(-1)^{in}\binom{2r}{i}u_{r-i}^{2n}\right) \ if\ r\ is\ odd.%
\end{cases}%
\end{equation*}
\end{theorem}

\begin{proof}
In this case after we expand $v_k$ to the power $2r$, we get:

\begin{equation*}
\sum\limits_{k=0}^{n}\binom{2n}{n+k}v_{k}^{2r}=\sum\limits_{k=0}^{n}\binom{2n%
}{n+k}
\left[\sum_{i=0}^{2r}\binom{2r}{i}\alpha^{(2r-i)k}\beta^{ik}\right].
\end{equation*}

Again using that $\alpha^k \beta^k=(-1)^k$ and the facts that $\binom{2r}{i}=%
\binom{2r}{2r-i}$, $(-1)^i=(-1)^{2r-i}$, for $i=0,1,2,\ldots,2r$,
we can turn the above into

\begin{eqnarray*}
\sum\limits_{k=0}^{n}\binom{2n}{n+k}v_{k}^{2r}= \sum\limits_{k=0}^{n}\binom{%
2n}{n+k}\left[(-1)^{rk}\binom{2r}{r}+\sum_{i=0}^{r-1}\left( -1\right) ^{ik}%
\binom{2r}{i}(\alpha^{2(r-i)k}+\alpha^{-2(r-i)k})\right].
\end{eqnarray*}

Commuting the two summations and using Proposition~\ref{cor1} the
above can be continued into

\begin{equation*}
\begin{array}{l}
\displaystyle \sum\limits_{k=0}^{n}\binom{2n}{n+k}v_{k}^{2r}= \left(\frac{1}{%
2}\binom{2r}{r}f(n,(-1)^r)+ \sum_{i=0}^{r-1}\binom{2r}{i}f(n,(-1)^i%
\alpha^{2(r-i)})\right)= \\
\\
\displaystyle \left( \frac{1}{2}\binom{2r}{r}f(n,(-1)^r)+ \sum_{i=0}^{r-1}%
\binom{2r}{i} [\tbinom{2n}{n}+(-1)^{in}\left(
(-1)^{r-i}\beta^{r-i}+(-1)^i\alpha^{r-i}\right)^{2n} ]\right)= \\
\\
\displaystyle
\begin{cases}
\displaystyle
\binom{2r}{r}2^{2n-1}+2^{2r-1}\tbinom{2n}{n}+\sum_{i=0}^{r-1}
(-1)^{in}\binom{2r}{i}v_{r-i}^{2n} \ if \ r\ is \ even. \\
\\
\displaystyle 2^{2r-1}\tbinom{2n}{n}+ \left(p^{2}+4\right)^{n}
\left(\sum_{i=0}^{r-1} (-1)^{in}\binom{2r}{i}u_{r-i}^{2n}\right)\
if \ r\ is
\ odd.%
\end{cases}%
\end{array}%
\end{equation*}
\end{proof}

As another application of Lemma~\ref{lem1}  to the second order
recurrence $\left\{ u_{n}\right\},$ we have the following
consequence.

\begin{proposition}
\label{thm1}For $n\geq 0$ and all even integers $m\geq 2$ and
$t\geq 0,$
\begin{equation}
\sum_{k=0}^{n}\tbinom{2n}{n+k}\left( v_{km}-v_{kt}\right) =\left(
p^{2}+4\right) ^{n}\left( u_{m/2}^{2n}-u_{t/2}^{2n}\right)
\label{5}
\end{equation}%
where $p$ is as before.
\end{proposition}

\begin{proof}
By Lemma~\ref{lem1} and the Binet formulas of the sequences
$\left\{
u_{n}\right\} $ and $\left\{ v_{n}\right\} $, we write the right side of (%
\ref{5}) as%
\begin{eqnarray*}
&&\sum_{k=0}^{n}\tbinom{2n}{n+k}\left( \alpha ^{km}+\beta
^{km}-\alpha
^{kt}-\beta ^{kt}\right) \\
&=&f\left( n,\alpha ^{m}\right) -f\left( n,\alpha ^{t}\right) \\
&=&\frac{1}{\alpha ^{mn}}\left( \alpha ^{m}+1\right)
^{2n}-\frac{1}{\alpha
^{tn}}\left( \alpha ^{t}+1\right) ^{2n} \\
&=&\left( \alpha ^{m/2}-\beta ^{m/2}\right) ^{2n}-\left( \alpha
^{t/2}-\beta
^{t/2}\right) ^{2n} \\
&=&\left( \alpha -\beta \right) ^{2n}\left(
u_{m/2}^{2n}-u_{t/2}^{2n}\right)
\end{eqnarray*}%
which completes the proof.
\end{proof}

One interesting consequence of Proposition \ref{thm1} is the
following.

\begin{proposition}
\label{cor2}For $n\geq 0,$
\begin{equation*}
\sum_{k=0}^{n}\tbinom{2n}{n+k}v_{2k}=\left( p^{2}+4\right) ^{n}.
\end{equation*}
\end{proposition}

\begin{proof} We take $m=2$, $t=0$ in Proposition~\ref{thm1}, and the proof
follows.
\end{proof}

\section{Representations of the powers of every integer and other comments}

To address the question we have raised in the introduction we
notice that as a result of Theorem~\ref{thm-p} for $r=1$ we obtain

\begin{equation*}
\sum\limits_{k=0}^{n}\binom{2n}{n+k}u_{k}^{2}=(p^2+4)^{n-1},\ \ n\in \mathbb{%
N}, \ p\in \mathbb{C}.
\end{equation*}

Suppose we would like to have a power of $7$ in the above equality, i.e. $%
p^2+4=7$. This can be accomplished if, for instance, $p=\sqrt{3}$.
This turns $\{u_k\}$ and $\{v_k\}$ into sequences of the form
$u_k=a_k+b_k\sqrt{3}
$ and $v_k=c_k+d_k\sqrt{3}$. Then the sequences $\{a_k\}$, $\{b_k\}$, $%
\{c_k\}$, and $\{d_k\}$ are uniquely determined by the recurrences

\begin{equation}  \label{genfibonacii}
\begin{array}{l}
a_0=b_0=0,a_1=1,b_1=0, \ a_{n+1}=3b_n+a_{n-1},\ b_{n+1}=a_n+b_{n-1},\ \ and%
\end{array}%
\end{equation}
\begin{equation}
\begin{array}{l}
c_0=2,d_0=0,c_1=1,d_1=0, \  \\
\\
c_{n+1}=3d_n+c_{n-1},\ d_{n+1}=c_n+d_{n-1}, \ n\in \mathbb{N}.%
\end{array}%
\end{equation}

Hence, for a generalized Fibonacci double sequence defined as in (\ref%
{genfibonacii}) we obtain a similar identity to (\ref{new1}):

\begin{equation*}
\sum\limits_{k=0}^{n}\binom{2n}{n+k}(a_k^{2}+3b_k^2)=7^{n-1},\ \
n\in \mathbb{N}.
\end{equation*}

\noindent Let us observe that $\sum\limits_{k=0}^{n}\binom{2n}{n+k}%
a_{k}b_{k}=0$ which implies that $a_{k}b_{k}=0$ for every $k\in
\mathbb{N}$.
This is a little surprising since the two sequences $\{a_{k}\}$ and $%
\{b_{k}\}$ seem to be increasing.

\noindent Theorem~\ref{thm-p} for $r=4$ gives

\begin{equation*}
\sum\limits_{k=0}^{n}\binom{2n}{n+k}F_{k}^{8}=\frac{1}{625} \left(
70\times 2^{2n-1}+7^{2n}+8(-1)^{n+1}4^{2n}+28\times
3^{2n}+56(-1)^{n+1}\right).
\end{equation*}

\noindent Certain congruences can be derived by use of these
identites. For example:

\begin{equation*}
\forall n\in \mathbb{N},\ \ 3^{2n}-4\left( -1\right) ^{n}+3\times
2^{2n}\equiv 0\ (mod\ 25)
\end{equation*}%
and

\begin{equation*}
\forall n\in \mathbb{N},\ \ 70\times
2^{2n-1}+7^{2n}+8(-1)^{n+1}4^{2n}+28\times
3^{2n}+56(-1)^{n+1}\equiv 0\ (mod\ 625).
\end{equation*}%
Other such identites might be derived by the interested reader.

\medskip
 \noindent AMS Classification Numbers: 11B37


\begin{thebibliography}{99}
\bibitem{C} L. Carlitz, \emph{Some classes of Fibonacci sums}, The Fibonacci
Quarterly, \textbf{16} (5) (1978), 411--426.

\bibitem{D} R A Dunlap, \emph{The Golden Ratio and Fibonacci Numbers}, World
Scientific Publishing Co. River Edge, NJ, 1997

\bibitem{P} P. Haukkanen, \emph{Formal power series for binomial sums of sequences
of numbers}, The Fibonacci Quarterly, \textbf{31} (1) (1993)
28--31.

\bibitem{H} V. E. Hoggatt, \emph{Advanced Problem H-88}, The Fibonacci Quarterly, (1968),
253.

\bibitem{V} S. Vajda, \emph{Fibonacci and Lucas numbers, and the Golden Section:
Theory and Applications}, John Wiley \& Sons, Inc. New York, 1989.

\end{thebibliography}
\end{document}